\documentclass[a4paper,12pt]{article}

\usepackage[utf8]{inputenc}
\usepackage{amstext}
\usepackage{amsfonts}
\usepackage{amsthm}
\usepackage{textcomp}
\usepackage{amssymb}
\usepackage{amscd}
\usepackage{amsmath}
\usepackage{xcolor}
\usepackage{tikz-cd}
\usepackage{caption}
\usepackage{enumitem}
\usepackage{theoremref}
\usepackage{mathtools}
\setlist[enumerate]{label=({\arabic*})}

\newtheorem{defn}{Definition}[subsection]
\newtheorem{thm}[defn]{Theorem}
\newtheorem{exmp}[defn]{Example}
\newtheorem{lem}[defn]{Lemma}
\newtheorem{rmk}[defn]{Remark}
\newtheorem{prob}[defn]{Problem}
\newtheorem{prop}[defn]{Proposition}
\newtheorem{cor}[defn]{Corollary}

\newtheorem{manualtheoreminner}{Theorem}
\newenvironment{manualtheorem}[1]{%
  \renewcommand\themanualtheoreminner{#1}%
  \manualtheoreminner
}{\endmanualtheoreminner}

\newtheorem{manualcorinner}{Corollary}
\newenvironment{manualcor}[1]{%
  \renewcommand\themanualcorinner{#1}%
  \manualcorinner
}{\endmanualcorinner}

\newcommand{\N}{\mathbb{N}}
\newcommand{\R}{\mathbb{R}}
\newcommand{\Z}{\mathbb{Z}}

\newcommand{\T}{\mathbb{T}}
\newcommand{\B}{{\mathcal B}}

\newcommand{\hc}[1]{{\mathcal {HC} {#1}}}
\newcommand{\hd}[1]{{\mathcal {HD} {#1}}}
\newcommand{\gh}[1]{{\mathcal {GH} {#1}}}
\newcommand{\U}{{\mathcal U}}

\newcommand{\uc}[1]{{\mathcal U}_{{C}}{#1}}
\newcommand{\ud}[1]{{\mathcal U}_{{D}}{#1}}

\newcommand{\ughp}[1]{{\mathcal U}_{{GH+}}{#1}}
\newcommand{\ughm}[1]{{\mathcal U}_{{GH-}}{#1}}

\DeclareMathOperator*\lowlim{\underline{lim}}
\DeclareMathOperator*\uplim{\overline{lim}}

\date{}
\begin{document}
\title{Generalized Hyperbolicity and Shadowing in $L^p$ spaces}

\author{Emma D'Aniello \\
\and  
Udayan B. Darji \\ 
\and 
Martina Maiuriello  \\}

\newcommand{\Addresses}{{
  \bigskip
  \footnotesize

  E.~D'Aniello, \\
  \textsc{Dipartimento di Matematica e Fisica,\\ Universit\`a degli Studi della Campania ``Luigi Vanvitelli",\\
  Viale Lincoln n. 5, 81100 Caserta, ITALIA} \\
  \textit{E-mail address: \em emma.daniello@unicampania.it} 
  \medskip

  U.B.~Darji,\\
  \textsc{Department of Mathematics,\\ University of Louisville,\\
  Louisville, KY 40292, USA}\\
  \textit{E-mail address: \em ubdarj01@louisville.edu}
  \medskip

  M.~Maiuriello,\\
  \textsc{Dipartimento di Matematica e Fisica,\\ Universit\`a degli Studi della Campania ``Luigi Vanvitelli",\\
  Viale Lincoln n. 5, 81100 Caserta, ITALIA}\\
  \textit{E-mail address: \em martina.maiuriello@unicampania.it}

}}

\maketitle

\begin{abstract}
It is rather well-known that hyperbolic operators have the shadowing property. In the setting of finite dimensional Banach spaces, having the shadowing property is equivalent to being hyperbolic. In 2018, Bernardes et al. constructed an operator with the shadowing property which is not hyperbolic, settling an open question. In the process, they introduced a class of operators which has come to be known as generalized hyperbolic operators. This class of operators seems to be an important bridge between hyperbolicity and the shadowing property. In this article, we show that for a large natural class of operators on $L^p(X)$ the notion of generalized hyperbolicity and the shadowing property coincide. We do this by giving sufficient and necessary conditions for a certain class of operators to have the shadowing property. We also introduce computational tools which allow construction of operators with and without the shadowing property. Utilizing these tools, we show how some natural probability distributions, such as the Laplace distribution and the Cauchy distribution, lead to operators with and without the shadowing property on $L^p(X)$. 
\end{abstract}

\let\thefootnote\relax\footnote{\date{\today} 
}
\newpage
\tableofcontents
\newpage

%##############################
\section{Introduction}
%##############################
Linear dynamics is a relatively recent area of mathematics which lies at the intersection of operator theory and dynamical systems. During the last two decades, a flurry of intriguing results have been obtained in this area concerning dynamical properties such as transitivity,  mixing, Li-Yorke, Devaney and distributional chaos, invariant measures, ergodicity and frequent hypercyclicity.  We refer the reader to books \cite{BernardesMessaoudiETDS2020} and \cite{GrosseErdmannMaguillot2011} for general information on the topic. 

Hyperbolic dynamics on manifolds is an important part of smooth dynamical systems. Indeed, some of the important questions in hyperbolic dynamics concern relationships between hyperbolicity, the shadowing property, expansivity and structural stability. It is rather well-known that hyperbolicity implies the shadowing property, expansivity and structural stability. Classical results of Smale \cite{Smale1967} and Walters \cite{Walters1978} show that the shadowing property and expansivity imply spectral decomposition and structural stability. Abdenur and Diaz \cite{AbdenurDiaz2007} showed   that, for generic $C^1$ homeomorphisms on closed manifolds, the shadowing property implies hyperbolicity in certain important contexts. Pilyugin and Tikhomirov \cite{PilyuginTikhomirov2010}  showed 
 that the Lipschitz shadowing property and structural stability are equivalent for $C^1$ homeomorphisms of closed smooth manifolds. This is just a glimpse of important works in this field.

Although hyperbolic dynamics of linear operator is  rather recent, there are some classical results from the 1960's where relationships between expansivity and spectrum of an operator were obtained. In particular, Eisenberg and Hedlund \cite{EisenbergHedlundPJM1970, HedlundPJM1971} showed that an invertible operator $T$ is uniformly expansive if and only if $\sigma_a (T)$, the approximate spectrum of $T$, does not intersect the unit circle $\T$. In 2000, Mazur \cite{MazurFDE2000} showed that  an invertible normal operator on a Hilbert space has the shadowing property if and only if it is hyperbolic. A detailed study of hyperbolicity, expansivity, the shadowing property and the spectrum of an operator on Banach space was initiated by Bernardes et al. in 2018 \cite{BernardesCiriloDarjiMessaoudiPujalsJMAA2018}. Among many results obtained there, an important question was settled, namely that there are operators with the shadowing property which are not hyperbolic. This result was proved by constructing a class of operators which have a weaker splitting than the usual splitting of hyperbolic operators. Cirilo et al. in a subsequent work named them  generalized hyperbolic operators. This class of operators seems to be the correct bridge between hyperbolicity and the shadowing property as it is evident by results in \cite{BernardesMessaoudiETDS2020}, \cite{BernadesMessaouidPAMS2020} 
and \cite{CiriloGollobitPujals2020}. That generalized hyperbolic operators have the shadowing property was shown in [Theorem A, \cite{BernardesCiriloDarjiMessaoudiPujalsJMAA2018}]. Bernardes and Messaoudi [Theorem 18, \cite{BernardesMessaoudiETDS2020}] gave a characterization of  weighted shifts 
which have the shadowing property. From this characterization one obtains that for the class of weighted shifts on $\ell^p(Z)$, generalized hyperbolicity is equivalent to the shadowing property. 
Before describing results in this article, we mention some seminal recent results. Bernardes and Messaoudi \cite{BernardesMessaoudiETDS2020} showed that a linear operator on a Banach space is hyperbolic if and only if it is expansive and has the shadowing property. In \cite{BernadesMessaouidPAMS2020} they also showed that all generalized hyperbolic operators are structurally stable. We also point out that general properties of generalized hyperbolic operators with applications are carried out in \cite{CiriloGollobitPujals2020}. In a very different direction from linear dynamics, relationship between hyperbolicity, expansivity and the shadowing property in the setting of noncompact spaces was carried out in \cite{LeeNguyenYang2018}.

In this article we explore the boundary between generalized hyperbolicity and the shadowing property. In particular, we show that for a large natural class of operators on $L^p(X)$ the notions of generalized hyperbolicity and the shadowing property coincide. More specifically, we start with a $\sigma$-finite dissipative measure space $(X, \B, \mu)$ and a nonsingular, invertible, bimeasurable transformation $f:X \rightarrow X$. We consider the composition operator $T_f:L^p(X) \rightarrow L^p(X)$ given by $T_f(\varphi) = \varphi \circ f$. If the Radon-Nikodym derivative of $\mu(f)$ with respect to $\mu$ is bounded below away from zero, then $T_f$ is a bounded operator. We assume such is the case for the Radon-Nikodym derivative of $\mu(f)$ and $\mu(f^{-1})$ with respect to $\mu$. Moreover, we assume that our measurable transformation satisfies the bounded distortion condition. Among this class of operators, we give necessary and sufficient conditions for an operator to have the shadowing property: Theorems SS, SN. Using the obtained characterization of the shadowing property, we conclude as a corollary that, in this particular class of operators, the shadowing property and generalized hyperbolicity coincide. In Theorem RN, we give computationally useful conditions which easily allow  construction of operators with and without shadowing property. In particular, 
we show how some natural probability distributions, such as the Laplace distribution and the Cauchy distribution, lead to operators with and without the shadowing property on $L^p(X)$. 

At this point we like to point out that a systematic study of composition operators in the setting of linear dynamics was initiated in \cite{BayartDarjiPiresJMAA2018} and \cite{BernardesDarjiPiresMM2020}. In \cite{BayartDarjiPiresJMAA2018}, necessary and sufficient conditions were given for a composition operator to be topologically transitive and mixing. Necessary and sufficient conditions for an operator to be Li-Yorke chaotic were given in \cite{BernardesDarjiPiresMM2020}. The motivation for the study of composition operators is to have a concrete but large class of operators which can be utilized as examples and counterexamples in linear dynamics. These types of operators include weighted shifts but the class is much larger than that. For example, it includes operators induced by measures on odometers \cite{BDDPOdometers}.

The paper is organized as follows. In Section~2, we give definitions and background results. In Section~3, we state our main results. In Section~4, we construct concrete examples. Section~5 consists of proofs, and Section~6 of open problems.
%##############################
\section{Definitions and Background Results}
%##############################
Given a Banach space $X$, by $S_X$ we denote the {\em unit sphere} of $X$, that is $S_X=\{x \in X\, : \, \Vert x \Vert =1 \}$.   If $T$ is a bounded operator on a Banach space $X$, then $\sigma(T),  \sigma_p(T), \sigma_a(T)$ and $\sigma_r(T)$ denote, respectively, the {\em spectrum}, the {\em point spectrum}, the {\em approximate point spectrum} and the {\em residual spectrum} of $T$, while $r(T)$ denotes the {\em spectral radius} of $T$ and it satisfies the {\em spectral radius formula} $r(T)=\lim_{n \rightarrow \infty} \Vert T^n \Vert ^{\frac{1}{n}}$.  In the sequel, as usual, ${\mathbb N}$ denotes the set of all positive integers and ${\mathbb N}_0={\mathbb N}\cup \{0\} .$ Moreover, ${\mathbb D}$ and ${\mathbb T}$ denote the open unit disk and the unit circle in the complex plane ${\mathbb C}$, respectively.

%@@@@@@@@@@@@@
\subsection{Weighted Shifts}
%@@@@@@@@@@@@@
Due to the importance of {\em weighted shifts} in the area of linear dynamics and operator theory, the study of their dynamical behavior has received special attention in recent years. We recall some preliminary definitions and results.

\begin{defn} 
Let $A= {\mathbb Z}$ or $A = {\mathbb N}$. Let $X=\ell^p(A)$, $1 \leq p < \infty$ or $X=c_0(A).$ Let  $w=\{w_n\}_{n \in A}$ be a bounded sequence of scalars, called {\em weight sequence}.  Then, the {\em weighted backward shift $B_w$ on $X$} is defined by \[B_w( \{x_n\}_{n \in A}) =\{w_{n+1}x_{n+1}\}_{n \in A}.\]
If $A= {\mathbb Z}$, the shift is called {\em bilateral}. If $A = {\mathbb N}$, then the shift is {\em unilateral}. 
A unilateral weighted backward shift is not invertible. On the other hand,  a bilateral $B_w$ is invertible if and only if  $\inf_{n \in \mathbb Z} \vert w_n \vert >0$.

\end{defn}

%@@@@@@@@@@@@@
\subsection{Expansivity}
%@@@@@@@@@@@@@
Expansivity is an important concept in hyperbolic dynamics. In the context of linear dynamics, various notions of expansivity have simpler formulations.  We use them as defined below. We refer the reader to \cite{BernardesCiriloDarjiMessaoudiPujalsJMAA2018} for a discussion of how they are obtained from the original definitions in the general setting.

\begin{defn}
An invertible operator $T$ on a Banach space $X$ is said to be {\em expansive} if for each $x \in S_X$ there exists $n \in \mathbb Z$  such that $\Vert T^n x\Vert \geq 2$.
\end{defn}

\begin{defn}
An invertible operator $T$ on a Banach space $X$ is said to be {\em uniformly expansive} if there exists $n \in \mathbb N$ such that \[z \in S_X \Longrightarrow \Vert T^nz \Vert \geq 2 \text{ or } \Vert T^{-n}z \Vert \geq 2.\]
\end{defn}

We point out that in the previous definitions, the number 2 can be replaced by any number $c>1$. 
In \cite{BernardesCiriloDarjiMessaoudiPujalsJMAA2018}, the authors characterize various types of expansivity for invertible operators on Banach spaces (\cite{BernardesCiriloDarjiMessaoudiPujalsJMAA2018}: Proposition 19) and, in particular, they also obtain a complete characterization of the notions of expansivity for weighted shifts (\cite{BernardesCiriloDarjiMessaoudiPujalsJMAA2018}: Theorem E). 

%@@@@@@@@@@@@@
\subsection{Shadowing} 
%@@@@@@@@@@@@@
As for expansivity, the concept of shadowing has a simplified formulation in the setting of linear dynamics. We use this formulation in our work.
\begin{defn}
Let $T:X \rightarrow X$ be an operator on a Banach space $X$. A sequence $\{x_n\}_{n \in \mathbb Z}$ in $X$ is called a {\em $\delta$-pseudotrajectory} of $T$, where $\delta >0$, if \[ \Vert Tx_n - x_{n+1} \Vert \leq \delta, \text{ for all $n \in \mathbb Z$.}\]
\end{defn}

The basic property of an operator related to the notion of a pseudotrajectory is {\em the shadowing property}:

\begin{defn}
Let $T:X \rightarrow X$ be an invertible operator on a Banach space $X$. Then $T$ is said to have the {\em shadowing property} if for every $\epsilon >0$ there exists $\delta >0$ such that every $\delta $-pseudotrajectory  $\{x_n\}_{n \in \mathbb Z}$ of $T$ is $\epsilon $-shadowed by a real trajectory of $T$, that is, there exists $x \in X$ such that \[ \Vert T^n x-x_n \Vert < \epsilon, \text{ for all $n \in \mathbb Z$.}\]
\end{defn}

We can define the notion of {\em positive shadowing} for an operator $T$ by replacing the set $\mathbb Z$ by $\mathbb N$ in the above definition. In such ``positive" case, $T$ 
does not need to be invertible.  

The following is an equivalent formulation of shadowing in the context of linear dynamics which one normally uses.
\begin{lem} [\cite{Pilyugin1999}] \label{LEM1}
An invertible operator $T$ on a Banach space $X$ has the shadowing property if and only if there is a constant $K>0$ such that, for every bounded sequence  $\{z_n\}_{n \in \mathbb Z}$ in $X$, there is a sequence $\{y_n\}_{n \in \mathbb Z}$ in $X$ such that \[ \sup_{n \in \mathbb Z} \Vert y_n \Vert \leq K \sup_{n \in \mathbb Z} \Vert z_n \Vert \hspace{0.3 cm}\text{ and } \hspace{0.3 cm} y_{n+1}=Ty_n + z_n, \text{ for all $n \in \mathbb Z$.}\] 
\end{lem}

In \cite{BernardesMessaoudiETDS2020}, Bernardes and Messaoudi establish the following characterization of shadowing for bilateral weighted backward shifts. 

\begin{thm} [\cite{BernardesMessaoudiETDS2020}: Theorem 18] \label{theoSHADBW}
Let  $X=\ell^p({\mathbb Z})$ $(1 \leq p < \infty)$ or $X=c_0({\mathbb Z})$ and consider a bounded weight sequence  $w=\{w_n\}_{n \in \mathbb Z}$ with $\inf _{n \in \mathbb Z} \vert w_n \vert >0$. Then, the bilateral weighted backward shift $B_w:X \longrightarrow X$ has the shadowing property if and only if one of the following conditions holds:
\begin{itemize}
\item[a)]{$\lim_{n \rightarrow \infty}( \sup _{k \in {\mathbb Z}} \vert w_k \cdots w_{k+n}\vert ^{\frac{1}{n}} )<1;$}
\item[b)]{$\lim_{n \rightarrow \infty}( \inf_{k \in {\mathbb Z}} \vert w_{k} \cdots w_{k+n}\vert^{\frac{1}{n}} )>1;$}
\item[c)]{$\lim_{n \rightarrow \infty}(\sup _{k \in {\mathbb N}} \vert w_{-k} \cdots w_{-k-n}\vert^{\frac{1}{n}} )<1$ and \\ $\lim_{n \rightarrow \infty}(\inf_{k \in {\mathbb N}} \vert w_{k} \cdots w_{k+n}\vert^{\frac{1}{n}} )>1.$  }
\end{itemize}
\end{thm}

%@@@@@@@@@@@@@
\subsection{Hyperbolicity and Generalized Hyperbolicity}
%@@@@@@@@@@@@@

A fundamental notion in linear dynamics is that of {\em hyperbolicity}.

\begin{defn}
An invertible operator $T$ is said to be {\em hyperbolic} if $\sigma (T) \cap {\mathbb T}= \emptyset.$
\end{defn}

It is known \cite{EisenbergHedlundPJM1970,HedlundPJM1971} that $T$ is uniformly expansive if and only if $\sigma_a(T)\cap {\mathbb T}=\emptyset$. Hence, every invertible hyperbolic operator is uniformly expansive and the converse, in general,  is not true \cite{BernardesCiriloDarjiMessaoudiPujalsJMAA2018, EisenbergHedlundPJM1970}. 

There is an equivalent useful formulation of hyperbolic operator which does not use the spectrum of the operator. 
It is classical that $T$ is hyperbolic if and only if there is a  splitting $X = X_s \oplus X_u$,
$T = T_s \oplus T_u$ (the {\em hyperbolic splitting} of $T$),
where $X_s$ and $X_u$ are closed $T$-invariant subspaces of $X$
(the {\em stable} and the {\em unstable subspaces} for $T$),
$T_s = T_{|_{X_s}}$ is a {\em proper contraction} (i.e., $\|T_s\| < 1$),
$T_u = T_{|_{X_u}}$ is invertible and it is a {\em proper dilation}
(i.e., $\|T_u^{-1}\| < 1$).

The above reformulation of hyperbolicity and the negative solution of the problem whether every operator with the shadowing property is hyperbolic (\cite{BernardesCiriloDarjiMessaoudiPujalsJMAA2018}: Theorem~B) led to the following notion of generalized hyperbolicity \cite{BernardesCiriloDarjiMessaoudiPujalsJMAA2018, BernardesMessaoudiETDS2020, CiriloGollobitPujals2020}.

\begin{defn} [\cite{CiriloGollobitPujals2020}: Definition 1] \label{DEFGH}
Let $T$ be an invertible operator on a Banach space $X$. If $X=M\oplus N$, where $M$ and $N$ are closed subspaces of $X$ with $T(M)\subset M$ and $T^{-1}(N) \subset N$,  and $T_{|_M}$ and $T^{-1}_{|_N}$ are proper contractions, then $T$ is said to be {\em generalized hyperbolic}.
\end{defn}
The following corollary ties this new concept to the shadowing property.
\begin{cor}[\cite{BernardesCiriloDarjiMessaoudiPujalsJMAA2018}: Corollary 8] \label{Cor1}  
Every invertible generalized hyperbolic operator $T$ on a Banach space $X$ has the  shadowing property.
\end{cor}
 It was long known that hyperbolicity and the shadowing properties are equivalent for special cases such as  in finite dimensional Banach spaces and for normal operators on Hilbert spaces \cite{MazurFDE2000, OmbachUIAM1994}. Recently, Bernardes and Messaoudi \cite{BernardesMessaoudiETDS2020} gave the precise conditions when they are equivalent.
\begin{thm}[\cite{BernardesMessaoudiETDS2020}: Theorem 1]
For any invertible operator $T$ on a Banach space $X$, the following are equivalent:
\begin{enumerate}
\item{$T$ is hyperbolic;}
\item{T is expansive and has the shadowing property.}
\end{enumerate}
\end{thm} 

Returning back to invertible bilateral weighted shifts, we summarize the known results in the following characterizations of hyperbolicity and generalized hyperbolicity. 

\begin{thm}
Let $X=\ell^p({\mathbb Z})$ $(1 \leq p < \infty )$ or $X=c_0({\mathbb Z}),$ and consider a weight sequence $w=\{w_n\}_{n \in \mathbb Z}$ with $\inf _{n \in {\mathbb Z}} \vert w_n \vert >0.$ Then,
\begin{enumerate}
    \item   $B_w$ is hyperbolic if and only if a) or b) of Theorem~\ref{theoSHADBW} are satisfied.
\item $B_w$ is generalized hyperbolic if and only if  it  has the shadowing property. 
\end{enumerate}
\end{thm}
\begin{proof}
Statement (1) is rather well-known. For example, see Remark 35 in \cite{BernardesCiriloDarjiMessaoudiPujalsJMAA2018}.
For Statement (2) we have already discussed above that generalized hyperbolic operators have the shadowing property. If a bilateral weighted backward shift has the shadowing property, then, using Theorem~\ref{theoSHADBW}, it can be easily shown that $B_w$ has a splitting as in Definition~\ref{DEFGH}, namely, we let 
\begin{align*}
   M &=  \{ \{x_n\}_{n \in \Z} \in \ell^p({\mathbb Z}): x_n = 0 \ \ \forall n \ge 0\} \\
   N &= \{ \{x_n\}_{n \in \Z} \in \ell^p({\mathbb Z}): x_n = 0 \ \ \forall n <0\}.
\end{align*}
\end{proof} 
We put these concepts in the following diagram to have a clear picture of the relationships between them. 
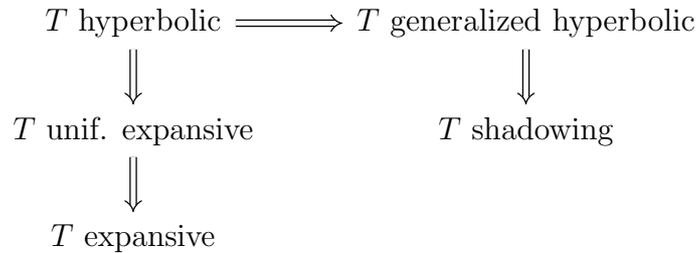
\[
\begin{tikzcd}
T \text{ hyperbolic}  \arrow[d, Rightarrow] \arrow[r, Rightarrow] &  T \text{  generalized hyperbolic} \arrow[d, Rightarrow] \\
T  \text{  unif. expansive} \arrow[d, Rightarrow] & T \text{ shadowing} \\
T \text{  expansive} \\
\end{tikzcd}
\]

%@@@@@@@@@@@@@
\subsection{Composition Operators}
%@@@@@@@@@@@@@
Our goal is to investigate the notions of generalized hyperbolicity and the shadowing property in the context of composition operators on $L^p$-spaces. We use the basic set up from \cite{BayartDarjiPiresJMAA2018, BernardesDarjiPiresMM2020}.

\begin{defn}\label{compodyn}
A {\em composition dynamical system} is a quintuple $(X,{\mathcal B},\mu, f, T_f)$ where
\begin{enumerate}
     \item $(X,{\mathcal B},\mu)$ is a $\sigma$-finite measure space, 
    \item $f : X \to X$ is an injective {\em bimeasurable transformation},
i.e., $f(B) \in {\mathcal B}$ and $f^{-1}(B) \in {\mathcal B}$ for every $B \in {\mathcal B}$,
\item there is  $c > 0$ such that
\begin{equation}\label{condition}
   \mu(f^{-1}(B)) \leq c \mu(B) \ \textrm{ for every } B \in {\mathcal B},
   \tag{$\star$}
\end{equation}
\item $T_f: L^p(X) \rightarrow L^p(X) $, $1 \le p <\infty$, is the {\em composition operator} induced by $f$, i.e.,
\[T_f : \varphi \mapsto \varphi \circ f.\] 
\end{enumerate}
\end{defn}
It is well-known that (\ref{condition}) guarantees that $T_f$ is a bounded linear operator. Moreover, if $f$ is surjective and $f^{-1}$ satisfies (\ref{condition}), then $T_{f^{-1}}$ is a well-defined bounded linear operator and $T^{-1}_f = T_{f^{-1}}$. We refer the reader to the book \cite{SinghManhas1993} for a detailed exposition on composition operators. 

%@@@@@@@@@@@@@
\subsection{Dissipative Systems and Bounded Distortion}
%@@@@@@@@@@@@@
Characterizing the shadowing property  and generalized hyperbolicity for composition operators seems complicated.  We are able to give an explicit characterization in the setting of a dissipative measure space. Even in this setting we need an additional condition. Below we give relevant definitions and recall how dissipative systems naturally arise from Hopf decomposition of general measurable systems. Throughout this paper all  measure spaces are $\sigma$-finite.

\begin{defn}\label{nullNS} A measurable transformation $f: X \rightarrow X$ on the measure space $(X, {\mathcal B}, \mu)$ is called {\em nonsingular} if, for any $B \in \mathcal B$, $\mu(f^{-1}(B))=0$ if and only if $\mu(B)=0$.
\end{defn}
We point out here that, if $f$ and $f^{-1}$ satisfy (\ref{condition}), then $f$ is nonsingular. Now we recall the Hopf Decomposition Theorem. 

\begin{thm}[Hopf, \cite{AaronsonMSM1997,Krengel1985}]
Let  $(X, {\mathcal B}, \mu)$ be a measure space and $f: X \rightarrow X$ be a nonsingular transformation. Then, $X$ is the union of two disjoint invariant sets ${\mathcal C}(f)$ and ${\mathcal D}(f)$, called the conservative and dissipative parts of $f$, respectively, satisfying the following conditions.
\begin{enumerate}
    \item For all $B \subseteq {\mathcal C}(f)$ with $\mu(B) >0$, there is $n >0$ such that $\mu(B \cap f^{-n}(
    B)) >0$.
    \item  ${\mathcal D}(f)$ is the pairwise disjoint union of $\{f^n(W)\}_{n \in \Z}$ for some $W \in {\mathcal B}$ i.e., ${\mathcal D}(f)= \dot{\cup}_{k=-\infty}^{+ \infty} f^k (W)$.
\end{enumerate}
\end{thm}
In above, the set $W$ is called a {\em wandering set of $f$}, i.e., $\{f^n(W)\}_{n \in \Z}$ are pairwise disjoint. In general, $\mu(W)$  does not have to  be finite.  
Based on Hopf Decomposition Theorem, we use the following definition of dissipative suitable for our purpose.
\begin{defn} \label{dissipcompodyn}
A measurable dynamical system $(X,{\mathcal B},\mu, f)$ is called a {\em dissipative system} if $X = \dot {\cup} _{k=-\infty}^{+ \infty} f^k (W)$ for some  $W \in \B$ with $0 < \mu (W) < \infty$. We will often say that the system is {\em generated by $W$.} 
\end{defn}
We now introduce a special type of dissipative system involving the notion of bounded distortion. It occurs naturally in various places, e.g., see \cite{VianaOliveira2016}. In the sequel, we let ${\mathcal B}(W) =\{ B \cap W, B \in {\mathcal B} \}.$

\begin{defn} \label{defnBD}
Let $(X,{\mathcal B},\mu, f)$ be a dissipative system generated by $W$.  We say that $f$ is of {\em bounded distortion on $W$} if there exists $K>0$ such that
\begin{equation}\label{conditionbd}
 \dfrac{1}{K} \mu(f^k(W))\mu(B) \leq \mu(f^k (B))\mu (W) \leq K \mu(f^k(W))\mu(B), \tag{$\Diamond$}
\end{equation}
for all $k \in \mathbb Z$ and  $B \in {\mathcal B}(W)$. 
In the case of above, we will say that $(X,{\mathcal B},\mu, f)$ is a {\em dissipative system of bounded distortion.}
\end{defn}
\begin{prop} \label{diststar}
Let $(X,{\mathcal B},\mu, f)$ be a dissipative system of bounded distortion generated by $W$. Then, the following are true.
\begin{enumerate}
    \item  There is a constant $H>0$ such that, for all $B \in {\mathcal B}(W)$ with $\mu(B)> 0$ and all $s, t \in \Z$, we have
\begin{equation} \label{generalbd}
  \dfrac{1}{H} \dfrac{\mu(f^{t+s}(W))}{\mu(f^s(W))} \leq \dfrac{\mu(f^{t+s} (B))}{\mu (f^s(B))} \leq H \dfrac{\mu(f^{t+s}(W))}{\mu(f^s(W))}. \tag{$\Diamond \Diamond$} 
\end{equation} 
    \item If $\sup \left \{\frac{\mu(f^{k-1}(W))}{\mu(f^{k}(W))}, \frac{\mu(f^{k+1}(W))}{\mu(f^{k}(W))}: k \in {\Z} \right \}$ is finite, then $f$ and $f^{-1}$ satisfy Condition (\ref{condition}).
\end{enumerate}  
\end{prop}
\begin{proof}
To prove the first part, we note that by Condition (\ref{conditionbd}) we have that  $\mu(B) =0$ if and only if $\mu(f^k(B)) =0$ for all $k \in \Z$. Hence, Condition (\ref{generalbd}) is well-defined for all $s, t \in \Z$. Let $K$ be the constant associated with the fact that $f$ is of bounded distortion on $W$. Then, for each  $s,t \in \Z$
\begin{eqnarray*}
 \dfrac{\mu(f^{t+s}(W))}{\mu(f^s(W))} & = & \dfrac{\mu(f^{t+s}(W))}{\mu(W)} \dfrac{\mu (W)}{\mu(f^s(W))} \\
& \leq & K^2 \dfrac{\mu(f^{t+s}(B))}{\mu(B)} \dfrac{\mu (B)}{\mu(f^s(B))} \\
& = & K^2 \dfrac{\mu(f^{t+s}(B))}{\mu(f^s(B)),}\\
\end{eqnarray*} 
and, analogously on the other side, we have that $\dfrac{\mu(f^{t+s}(W))}{\mu(f^s(W))} \geq \dfrac{1}{K^2} \dfrac{\mu(f^{t+s}(B))}{\mu(f^s(B)).}$
Setting $H=K^2$ completes the proof of the first part. 

For the second part, we will show that $f$ satisfies Condition (\ref{condition}). The proof for $f^{-1}$ is analogous. Let $M=\sup \left \{\frac{\mu(f^{k-1}(W))}{\mu(f^{k}(W))}, \frac{\mu(f^{k+1}(W))}{\mu(f^{k}(W))}: k \in {\Z} \right \}$. Let $A \in {\mathcal B}$ and set $A_k = A \cap f^k(W)$. By the  countable additivity property of measures, it suffices to show Condition (\ref{condition}) for $A_k$. If $\mu (A_k)=0$, then applying  (\ref{conditionbd}) to $B = f^{-k}(A_k) \subseteq W$ and $k$, we have that $\mu(f^{-k}(A_k)) =0$. Hence $\mu(f^l (f^{-k}(A_k)) =0$ for all $l \in \Z$ and, in particular, $\mu (f^{-1}(A_k))= 0$. For $\mu (A_k) > 0$,  we apply the right side of Condition ({\ref{generalbd}}) to 
$B= f^{-k}(A _k) \subseteq W$, $s=k$ and $t=-1$, we obtain that 
\[\dfrac{\mu(f^{-1} (A_k))}{\mu(A_k)} \leq H \dfrac{\mu(f^{-1+k}(W))}{\mu(f^k(W))} \leq H M.\]
Letting $c = HM$, we have that   
\[\mu(f^{-1} (A_k)) \leq c \mu(A_k), \ \ \forall k \in \Z, \]
verifying Condition (\ref{condition}). 
\end{proof}
For the following, $\dfrac{d\mu(f^{k})}{d\mu}$ denotes the Radon-Nikodym derivative of $\mu(f^{k})$ with respect to $\mu$.
\begin{prop}[Bounded RN Condition] \label{PROPBRN} Let
$(X,{\mathcal B},\mu, f)$ be a dissipative  system generated by $W$. Let ${\rho}_{k} = \dfrac{d\mu(f^{k})}{d\mu}$, 
$m_{k} = \underset{ x \in W}{\mathrm{ess\,inf}} \  {\rho}_{k} (x) $, and    $M_{k} = \underset{ x \in W}{\mathrm{ess\,sup}} \  {\rho}_{k} (x)$. 
If 
$\left \{\frac{M_k}{m_k}    \right \} _{k \in \Z}$ is bounded,
then $f$ is of bounded distortion on $W$.
\end{prop}
\begin{proof}
Let $K$ be a bound on
$\left \{\frac{M_k}{m_k}    \right \} _{k \in \Z}$. We prove that Condition $(\Diamond)$ holds. If $B \in {\mathcal B}(W)$ with $\mu(B) =0$, then Condition $(\Diamond)$ clearly holds as all Radon-Nikodym derivatives are bounded above. Hence, let us consider the case $\mu(B) >0$.
For every $k \in {\Z}$ and $B \in {\mathcal B}(W)$, 
\[\mu(f^{k}(B)) = \int_{B} {\rho}_{k} d \mu \leq \int_{B} M_k d \mu = M_k \mu(B) \]
and 
\[\mu(f^{k}(W)) =   \int_{W} {\rho}_{k} d \mu \geq  \int_{W} m_k   d \mu =   m_k \mu(W). \] 
Dividing the two inequalities we get
\[\frac{\mu(f^{k}(B))}{\mu(f^{k}(W))} \leq \frac{M_k}{m_k} \frac{\mu(B)}{\mu(W)} \le K \frac{\mu(B)}{\mu(W)} \]
and, on the other side,
\[\frac{\mu(f^{k}(B))}{\mu(f^{k}(W))} \geq \frac{m_k}{M_k} \frac{\mu(B)}{\mu(W)} \ge \frac{1}{K} \frac{\mu(B)}{\mu(W)}. \] Putting them together, we have that 
\[\frac{1}{K} \frac{\mu(f^k(W))}{\mu(W)} \leq \frac{\mu(f^{k}(B))}{\mu(B)} \leq K \frac{\mu(f^{k}(W))}{\mu(W)},\]
i.e., Condition $(\Diamond)$ holds.
\end{proof}

%##############################
\section{Main Results and Examples}
%###############################
%@@@@@@@@@@@@@
\subsection{Shadowing Results}
%@@@@@@@@@@@@@
Throughout this subsection, we assume that $T_f$ is a well-defined invertible operator, i.e., functions $f$ and $f^{-1}$ satisfy (\ref{condition}).

The following three conditions are essential in the description of characterization of the shadowing property.
As the formulas are long, we give them names to avoid writing them repeatedly. 
\begin{defn}
Let $(X,{\mathcal B},\mu, f)$ be a measurable  system.
We say that {\em Conditions $\hc{}$, $\hd{}$ and $\gh{}$ hold}, respectively, when the following are true: 
\begin{equation}\label{hc}
\uplim_{n \rightarrow \infty} \sup _{k \in {\mathbb Z}} {\left (\frac{\mu(f^{k}(W))}{\mu(f^{k+n}(W))}\right )}^{\frac{1}{n}} <1 \tag*{${\hc{}}$}   
\end{equation}
\begin{equation}\label{hd}
 \lowlim_{n \rightarrow \infty} \inf _{k \in {\mathbb Z}} {\left (\frac{\mu(f^{k}(W))}{\mu(f^{k+n}(W))} \right)}^{\frac{1}{n}} >1 \tag*{$\hd{}$} 
\end{equation}
\begin{equation}{\label{gh}
  \uplim_{n \rightarrow  \infty} \sup _{k \in -{\mathbb N}_{0}} { \left (\frac{\mu(f^{k-n}(W))}{\mu(f^{k}(W))} \right )}^{\frac{1}{n}} < 1 
  \ \  \& \\ 
\lowlim_{n \rightarrow  \infty} \inf_{k \in {\mathbb N_{0}} } {\left (\frac{\mu(f^{k}(W))}{\mu(f^{k+n}(W))}\right )}^{\frac{1}{n}} >1 \tag*{$\gh{}$}}
\end{equation}
\end{defn}

We begin by sufficient conditions on the measurable system $(X,{\mathcal B},\mu, f)$ which guarantee the shadowing property of $T_f$.
\begin{manualtheorem}{SS}[Shadowing Sufficiency]\label{thmSS}
Let $(X,{\mathcal B},\mu, f)$ be a dissipative system of bounded distortion generated by $W$. Then the following hold.
\begin{enumerate}
     \item If Condition~$\hc{}$ is satisfied, then $T_f$ is a contraction. 
    \item If Condition~$\hd{}$ is satisfied, then $T_f$ is a dilation.
    \item If Condition~$\gh{}$ is satisfied, then $T_f$ is a generalized hyperbolic operator.
    \end{enumerate} 
    Hence, $T_f$ has the shadowing property in all three cases.
\end{manualtheorem}
Below we prove the necessary condition for shadowing. 
\begin{manualtheorem}{SN}[Shadowing Necessity]\label{thmSN}
Let $(X,{\mathcal B},\mu, f)$ be a dissipative system of bounded distortion generated by $W$. If the composition operator $T_f$ has the shadowing property then one of conditions~$\hc{}$, $\hd{}$ or $\gh{}$ holds.
\end{manualtheorem}
Putting Theorems~\ref{thmSS} and Theorem~\ref{thmSN} together, we have the following characterization of shadowing. 
\begin{manualcor}{SC}[Shadowing Characterization]\label{CorSC}
Let $(X,{\mathcal B},\mu, f)$ be a dissipative system of bounded distortion generated by $W$.  Then the following are equivalent.
\begin{enumerate}
     \item The composition operator $T_f$ has the shadowing property. 
    \item One of Conditions~$\hc{}$, $\hd{}$ or $\gh{}$ holds.
    \end{enumerate} 
\end{manualcor}
\begin{manualcor}{GH}[Generalized Hyperbolic Characterization]\label{CorGH}
Let $(X,{\mathcal B},\mu, f)$ be a dissipative system of bounded distortion. Then, the following are equivalent.
\begin{enumerate}
    \item{The composition operator $T_f$ is generalized hyperbolic.}
    \item{The composition operator $T_f$ has the shadowing property.}
\end{enumerate} 
\end{manualcor} 
\begin{proof}
We recall that every generalized hyperbolic operator has the shadowing property \cite{BernardesCiriloDarjiMessaoudiPujalsJMAA2018, CiriloGollobitPujals2020}. Hence, (1) implies (2). That (2) implies (1) follows from applying Theorem~\ref{thmSN} first and then Theorem~\ref{thmSS}.
\end{proof}
The following reformulation of Theorem~\ref{CorSC} will be a useful tool for giving explicit examples of composition operators with various properties. 
\begin{manualtheorem}{RN}\label{thmRN}
Let
$(X,{\mathcal B},\mu, f)$ be a dissipative system generated by $W$, ${\rho}_{k} = \dfrac{d\mu(f^{k})}{d\mu}$, $m_{k} = \underset{ x \in W}{\mathrm{ess\,inf}} \  {\rho}_{k} (x)
$ and $M_{k} = \underset{ x \in W}{\mathrm{ess\,sup}} \  {\rho}_{k} (x) $.
Furthermore, assume that 
$\left \{\frac{M_k}{m_k}    \right \} _{k \in \Z}$ is bounded. Then, the following are equivalent.
\begin{enumerate}
     \item The composition operator $T_f$ has the shadowing property. 
    \item One of the following properties hold.
    \end{enumerate} 
    \begin{equation}\label{contractionRN}
\uplim_{n \rightarrow \infty} \sup _{k \in {\mathbb Z}} {\left (\frac{M_k}{m_{k+n}}\right )}^{\frac{1}{n}} <1 \tag{${\mathcal {RNC} }$}   
\end{equation}
\begin{equation}\label{dialationRN}
 \lowlim_{n \rightarrow \infty} \inf _{k \in {\mathbb Z}} {\left (\frac{M_k}{m_{k+n}} \right)}^{\frac{1}{n}} >1 \tag{${\mathcal {RND} }$} 
\end{equation}
\begin{equation}\label{GHRN}
  \uplim_{n \rightarrow  \infty} \sup _{k \in -{\mathbb N}_{0}} { \left (\frac{M_{k-n}}{m_{k}} \right )}^{\frac{1}{n}} <1
 \   \ \& \ \ 
\lowlim_{n \rightarrow \infty} \inf_{k \in {\mathbb N}_{0} } {\left (\frac{M_k}{m_{k+n}}\right )}^{\frac{1}{n}} >1 \tag{${\mathcal {RNGH}}$}
\end{equation}
Moreover, Conditions ~\ref{contractionRN}, \ref{dialationRN}, \ref{GHRN} imply that $T_f$ is a contraction, a dilation, a generalized hyperbolic operator, respectively. \end{manualtheorem}
\begin{rmk}\label{rmkthmrn}
It will follow from the proof of Theorem~\ref{thmRN} that, in Condition~\ref{contractionRN}, Condition~\ref{dialationRN},
and Condition~\ref{GHRN}, one can exchange $M$ for $m$ and the theorem still holds. 
\end{rmk}
%@@@@@@@@@@@@@
\subsection{Shadowing Examples}
%@@@@@@@@@@@@@
Next, we show how composition operators with various properties can be constructed with ease using standard measures and probability distributions on ${\mathbb R}$.  For the next four examples, we will be working with $X = {\mathbb R}$, ${\cal B}$ the collection of Borel subsets of ${\mathbb R}$, and $f(x) = x+1$. Note that, independent of the measure $\mu$ we choose on  ${\mathbb R}$, we get a  dissipative system generated by $W = [0,1)$. Moreover, all of our $\mu$ will be given by a density, i.e.,
\[\mu(B) = \int _B h d\lambda,\]
where $\lambda$ is the Lebesgue measure on ${\mathbb R}$ and $h$ is some non-negative Lebesgue integrable function. 
As \[ \frac{d (\mu f^i)}{d\lambda} = \frac{d (\mu f^i)}{d\mu} \cdot  \frac{d\mu }{d\lambda} 
\]
and 
\[ \frac{d (\mu f^i)}{d\lambda} (x) = h(x+i), \ \ \ \ \ \ \  \ \ \ \ \ \frac{d\mu }{d\lambda} (x) = h(x),
\]
we have that \[ \frac{d (\mu f^i)}{d\mu} (x) = \frac{h(x+i)}{h(x)}.
\]

\begin{exmp}[Contraction $T_f$ ] \label{exmpcontraction}
Let $\mu$ be the measure whose density is $h(x) = e^ x$. Then, $\frac{d (\mu f^i)}{d\mu} (x) = \frac{h(x+i)}{h(x)} = e^i$. Applying Theorem~\ref{thmRN}, we obtain that 
\[
\uplim_{n \rightarrow \infty} \sup _{k \in {\mathbb Z}} {\left (\frac{M_k}{m_{k+n}}\right )}^{\frac{1}{n}} = \uplim_{n \rightarrow \infty} \sup _{k \in {\mathbb Z}} {\left (\frac{e^k}{e^{k+n}}\right )}^{\frac{1}{n}} =\frac{1}{e},\]
implying that $T_f$ is a contraction.
\end{exmp}
\begin{exmp} [Dilation $T_f$ ]\label{exmpdilation}
An analogous calculation to the above shows that, if we let  $\mu$ be a measure whose density is $e^{-x}$, then $T_f$ is a dilation.
\end{exmp}
\begin{exmp}[Generalized Hyperbolic $T_f$] \label{Lap}
In this example, we use Laplace distribution. Recall that the Laplace distribution is defined by the following probability density function \[h(x, b, \lambda)=\frac{1}{2b} e^{-\frac{\vert x- \lambda\vert}{b}},\] where $\lambda \in \R$ and $b>0$ are two parameters. For the sake of simplicity, we use the standard Laplace distribution, i.e.,  $\lambda=0$ and $b=1$. Hence, we let $\mu$ be the probability measure whose density is given by \[h(x)=\frac{1}{2} e^{-\vert x\vert}.\]
Using the fact that 
$\frac{d (\mu f^i)}{d\mu} (x) =  \frac{h(x+i)}{h(x)} = \frac{e^{-|x+i|}}{e^{-|x|} }$, we have 
\[ M_i =  e^{-i} \ \ \ \  m_i = e^{-i} \ \ \ \ \forall i \ge 0\]
\[
M_i =  e^{2+i} \ \ \ \  m_i =  e^{i} \ \ \ \ \forall i < 0.\]
For $n \in {\mathbb N}$, we have that
\[ \frac{M_k}{m_{k+n}} = e^n \   \mbox {  for }  k \ge 0 \ \ \ \ {and} \ \ \ \  \frac{M_{k-n}}{m_{k}} = e^{2-n} \   \mbox {  for }  k \leq 0.
\] 
Using the above estimates, it is readily verified that Condition~\ref{GHRN} holds and, hence, $T_f$ is generalized hyperbolic.
\end{exmp}
Our next example shows that our techniques can also be used to show that certain operators $T_f$ do not have the shadowing property. 

\begin{exmp}[Non-Shadowing $T_f$]
The standard Cauchy distribution is a continuous distribution on $\R$, defined by the following probability density function
\[ h(x)= \frac{1}{\pi(1+x^2)}. \]
As earlier, let $\mu$ be the probability measure whose density is $h$. We will show that none of Condition~\ref{contractionRN}, Condition~\ref{dialationRN},
nor Condition~\ref{GHRN} is satisfied, yielding that $T_f$ does not have the shadowing property. Indeed, we have that
\[\frac{d (\mu f^i)}{d\mu} (x) =  \frac{h(x+i)}{h(x)} = \frac{1+x^2}{1+(x+i)^2}.\]

This time calculating the exact values of $M_i$ and $m_i$ is a bit complicated.  However, we will find appropriate bounds on $M_i$'s and $m_i$'s and this will suffice. Note that, for $ 0 \le x \le 1$ and $i \ge 0$, we have 

\[ \frac{1}{1+(i+1)^2} \le \frac{1+x^2}{1+(x+i)^2}\le \frac{2}{1+i^2},
\]
implying 
\[ M_i  \le  \frac{2}{  1+i^2} \ \ \  \ \ \ \  \  m_i \ge  \frac{1}{ 1+ (i+1)^2} \ \ \ \ \forall i \ge 0.\]
Similarly, for $ 0 \le x \le 1$ and $i \leq 0$, we have 

\[ \frac{1}{1+i^2} \le \frac{1+x^2}{1+(x+i)^2}\le \frac{2}{1+(i+1)^2},
\]
implying
\[ M_i \ge  \frac{1}{ 1+i^2} \ \ \  \ \ \ \  \  m_i \le  \frac{2}{ 1+ (i+1)^2} \ \ \ \ \forall i \leq 0.\]
As Condition~\ref{contractionRN} implies the left half of Condition~\ref{GHRN} and 
Condition~\ref{dialationRN} implies the right half of Condition~\ref{GHRN},
it suffices to prove that both limits fail in Condition~\ref{GHRN}, in order to conclude that $T_f$ does not have the shadowing property. Observe that, for  $n \in {\mathbb N}$, we have
\[ \frac{M_k}{m_{k+n}} \le \frac{2[1+ (k+n+1)^2]}{1+k^2} \le 2 [1+ (k+n+1)^2],  \  \  \ \   k \ge 0 \] and
\[  \frac{M_{k-n}}{m_{k}} \ge  \frac{1+(k+1)^2}{2[1+ (k-n)^2]} \ge \frac{1}{2[1+ (k-n+1)^2]}, \  \  \ \   k  \leq 0.
\] 
Hence,
\[   \ \ 
\lowlim_{n \rightarrow \infty} \inf_{k \in {\mathbb N}_{0} } {\left ( \frac{M_k}{m_{k+n}} \right )}^{\frac{1}{n}}  \le \lowlim_{n \rightarrow \infty} \inf_{k \in {\mathbb N}_{0} } {\left\{ 2 [1+ (k+n+1)^2 ]\right \}}^{\frac{1}{n}} 
=1,\]
and 
\[  \uplim_{n \rightarrow  \infty} \sup _{k \in -{\mathbb N}_{0}} { \left (\frac{M_{k-n}}{m_{k}} \right )}^{\frac{1}{n}} \ge \uplim_{n \rightarrow  \infty} \sup _{k \in -{\mathbb N}_{0}} 
{ \left \{\frac{1}{ 2[1+ (k-n+1)^2]} \right \}}^{\frac{1}{n}} =1,
\]
verifying that both parts of Condition~\ref{GHRN} fail and completing the proof.
\end{exmp}
We end this subsection of examples by commenting that our methods are flexible enough to handle a large class of examples. For example, if we want to work in higher dimensions, we may take $X= {\R}^{2}$, $f(x,y)= (x,y) + (1,0)$ and $W=[0,1[ \times {\R}$. Then, taking different types of 2-dimensional joint density functions, we can obtain $T_f$ with various properties as in our 1-dimensional examples. 

%@@@@@@@@@@@@@@@@@@
\section{Shadowing Proofs}
%@@@@@@@@@@@@@@@@@@@@

Throughout this section, $(X,{\mathcal B},\mu, f)$ is a  dissipative system and $T_f$ is the associated invertible composition operator on $L^p(X)$. 

%@@@@@@@@@@@@@@@@@@@@
\subsection{Proof of Theorem~\ref{thmSS}}
%@@@@@@@@@@@@@@@@@@@@

We prove a series of propositions which lead to the proof of Theorem~\ref{thmSS}. We first introduce some notation and terminology to facilitate our proofs.  

\begin{defn}
Let $(X,{\mathcal B},\mu, f)$ be a dissipative  system generated by $W$ and $\varphi \in L^p(X)$. 

Then, $\varphi = \varphi_+ + \varphi_-$, where 
\[\varphi_+(x)= \left\{ \begin{array}{ll}
0  & \mbox{ if } x \in \cup_{k=0}^{\infty}f^{k}(W)\\
\varphi(x)   & \mbox{ otherwise, } \\
\end{array}
\right.\] 

and, similarly, $\varphi_ -$ is zero on  $\cup_{k=1}^{\infty}f^{-k}(W)$ and $\varphi$ elsewhere. 

Let $L_+ = \{\varphi_+: \varphi \in L^p(X)\}$ and  $L_-= \{\varphi_-: \varphi \in L^p(X)\}$. We note that $L^p(X) = L_+ \oplus L_-$ and $T_f(L_+) \subseteq L_+$ and $T_f^{-1}(L_-) \subseteq L_-$. 
\end{defn}
\begin{defn}
Let $(X,{\mathcal B},\mu, f)$ be a dissipative  system generated by $W$ and  $K, t>0$.  Let $\uc{(K, t)}$ and  $\ud{(K, t)}$ be the set of all $\varphi \in L^p(X)$ which satisfy the following conditions, respectively: 
\begin{equation}\label{uc}
\sup _{k \in {\mathbb Z}} {\left (\frac{\int_X \vert \varphi \vert ^p \circ f^{-k} d\mu}{\int_X \vert \varphi \vert ^p \circ f^{-(k+n)} d\mu}\right )} \leq K t^n  \ \ \  \ \ \forall n \in \N \tag*{$\uc{}$}   
\end{equation}
\begin{equation}\label{ud}
\inf _{k \in {\mathbb Z}} {\left (\frac{\int_X \vert \varphi \vert ^p \circ f^{-k} d\mu}{\int_X \vert \varphi \vert ^p \circ f^{-(k+n)} d\mu} \right)} \geq K t^n  \ \ \  \ \ \forall n \in \N \tag*{$\ud{}$} 
\end{equation}
We let $\ughp{(K,t)}$ and $\ughm{(K,t)}$ consist of those $\varphi$ in $L_+$ and $L_-$, respectively, which satisfy the following conditions:
\begin{equation*}\label{ughp}
   \sup _{k \in -{\mathbb N}_{0}} \left ( \frac{\int_X  \vert \varphi \vert ^p \circ f^{-(k-n)}  d\mu}{\int_X  \vert \varphi \vert ^p \circ f^{-k}  d\mu} \right ) \leq K t^n  \ \ \  \ \ \forall n \in \N \tag*{$\ughp{}$}
\end{equation*}
\begin{equation*}\label{ughm}
 \inf_{k \in {\mathbb N}_{0} } {\left (\frac{\int_X  \vert \varphi \vert ^p \circ f^{-k}  d\mu}{\int_X  \vert \varphi \vert ^p \circ f^{-(k+n)}  d\mu}\right )} \geq K \frac{1}{t^n}  \ \ \  \ \ \forall n \in \N \tag*{$\ughm{}$}.\\
\end{equation*}
\end{defn}
The next simple fact follows from the definitions of $\uplim$ and $\lowlim$.
\begin{prop}\label{propbasic}
Let $\{a_n\}_{n \in \N}$ be a sequence of non-negative real numbers and $t >0$.  Then, the following hold.
\begin{enumerate}
    \item {If $\uplim _{n \rightarrow \infty} a_n ^{\frac{1}{n}}<t,$ then there exists $K>0$ such that $a_n \leq Kt^n$ for every $n \in \N$.}
    \item {If $\lowlim _{n \rightarrow \infty} a_n ^{\frac{1}{n}}>t,$ then there exists $K>0$ such that $a_n \geq Kt^n$ for every $n \in \N$.}
\end{enumerate} 
\end{prop}
\begin{prop} \label{proprcW}
Let $(X,{\mathcal B},\mu, f)$ be a dissipative system generated by $W$. Then
\begin{enumerate}
\item  $\hc{}$ holds $\Leftrightarrow$ $\chi_W \in \uc{(K, t)}$ for some $K >0$ and $t<1$. 
\item $\hd{}$ holds $\Leftrightarrow$ $\chi_W \in \ud{(K,t)}$ for some $K >0$ and  $t>1$. 
\item  $\gh{}$ holds $\Leftrightarrow$ there exist $K >0$ and  $t<1$ such that $\chi_W \in \ughm{(K,t)}$ and $\chi_{f^{-1}(W)} \in \ughp{(K,t)}$.
\end{enumerate}
\end{prop}
\begin{proof} We first prove (1).

($\Rightarrow$) Suppose that $\hc{}$ holds. Applying Proposition \ref{propbasic} to the sequence $a_n=\sup _{k \in \Z} \frac{\mu(f^k(W))}{\mu (f^{k+n}(W))}$ and by the fact that $\mu(f^{s}(B))=\int_X \vert \chi_{B} \vert ^p \circ f^{-s} d\mu$, for  $s \in \Z$, $B \in \B$, we have that $\chi _W \in \uc{(K, t)}$ for some $K >0$ and $t<1$. 

($\Leftarrow$) Let $K >0$ and $t<1$ be such that $\chi _W \in \uc{(K, t)}$, that is,
\[ \sup _{k \in {\mathbb Z}} {\left (\frac{\int_X \vert \chi _W\vert ^p \circ f^{-k} d\mu}{\int_X \vert \chi _W\vert ^p \circ f^{-(k+n)} d\mu}\right )} \leq K t^n. \] 
Then, 
\begin{eqnarray*}
\uplim _{n \rightarrow \infty} \sup _{k \in \Z} \left(\frac{\mu(f^k(W))}{\mu (f^{k+n}(W))}\right)^{\frac{1}{n}} &=& \uplim _{n \rightarrow \infty}\sup _{k \in {\mathbb Z}} {\left (\frac{\int_X \vert \chi _W\vert ^p \circ f^{-k} d\mu}{\int_X \vert \chi _W\vert ^p \circ f^{-(k+n)} d\mu}\right )}^{\frac{1}{n}} \\
& \leq &  \uplim _{n \rightarrow \infty} K^{\frac{1}{n}}t \\
& = & t <1,
\end{eqnarray*}
i.e. condition $\hc{}$ holds.

(2) This proof is analogous to the proof of (1).

(3) 
($\Rightarrow$) Suppose that $\gh{}$ holds. We note that in $\gh{}$ we can replace $W$ by $f^{-1}(W)$ and the condition still holds. We use this in the first part of $\gh{}$.
Applying Proposition \ref{propbasic}  as before, we obtain $K >0$ and $0 < t <1$ such that, for every $n \in \N$, 
\[\sup_{k \in - {\mathbb N}_{0}} \frac{\mu(f^{k-n}(f^{-1}(W)))}{\mu (f^{k}(f^{-1}(W)))} \leq K t^{n}\] 
and 
\[\inf _{k \in  {\mathbb N}_{0}} \frac{\mu(f^{k}(W))}{\mu (f^{k+n}(W))} \geq  K {\frac{1}{t^n}}.\]
As
\[\inf _{k \in  {\mathbb N}_{0}}\frac{\int_X \vert \chi _W\vert ^p \circ f^{-k} d\mu}{\int_X \vert \chi _W\vert ^p \circ f^{-(k+n)} d\mu} = \inf _{k \in  {\mathbb N}_{0}} \frac{\mu(f^{k}(W))}{\mu (f^{k+n}(W))}  \ge  K {\frac{1}{t^n}},\]
we have that $\chi_W \in \ughm{(K,t)}$. 
Similarly, as  
\begin{eqnarray*}
 \sup_{k \in - {\mathbb N}_{0}}\frac{\int_X \vert \chi _{f^{-1}(W)}\vert ^p \circ f^{-(k-n)} d\mu}{\int_X \vert \chi _{f^{-1}(W)} \vert ^p \circ f^{-k} d\mu} =  \sup_{k \in - {\mathbb N}_{0}} \frac{\mu(f^{k-n}(f^{-1}(W)))}{\mu (f^{k}(f^{-1}(W)))} \leq K t^{n}  ,
\end{eqnarray*}
we have that $\chi_{f^{-1}(W)} \in \ughp{(K,t)}$. \\

($\Leftarrow$) This proof is straightforward and follows as in the proof of (1).
\end{proof}

Next proposition follows from the  definitions. 
\begin{prop}\label{prop:compfj} The following are true.
\begin{itemize}
    \item Let $\U (K,t) \in \{\uc{(K,t)}, \ud{(K,t)}\}$. If $\varphi \in \U(K,t)$ then $\varphi \circ f^j \in \U(K,t)$,  $j \in \Z$. 
    \item If $\varphi \in \ughp{(K,t)}$ then $\varphi \circ f^j \in\ughp{(K,t)}$,  $j \ge 0$. 
    \item If $\varphi \in \ughm{(K,t)}$ then $\varphi \circ f^j \in\ughm{(K,t)}$,  $j \le 0$. 
\end{itemize}
For the sake of notational convenience, for the next two propositions, we let  $\U (K,t) \in \{\uc{(K,t)}, \ud{(K,t)}, \ughp{(K, t)},\ughm{(K, t)}\}$. The first one simply follows from the definitions.
\end{prop}
\begin{prop}\label{prop:scalar}
If $\varphi \in \U(K,t)$ and $a \in \R \setminus \{0\}$, then  $a \cdot \varphi \in \U(K,t)$.
\end{prop}
\begin{prop} \label{finitesum}
If $\varphi _1, \varphi _2 \in \U(K,t)$ with disjoint supports, then $\varphi _1 + \varphi _2 \in \U(K,t)$.
\end{prop}
\begin{proof} 
We do the proof for $\U(K,t) = \uc{(K,t)}$. The proofs for the rest are analogous. 
Let $\varphi_1$ and $\varphi_2$ be elements of  $\uc{(K,t)}$, that is,
\[\sup _{k \in {\mathbb Z}} {\left (\frac{\int_X \vert \varphi_i \vert ^p \circ f^{-k} d\mu}{\int_X \vert \varphi_i \vert ^p \circ f^{-(k+n)} d\mu}\right )} \leq Kt^n,\]   
for $i=1,2$, with disjoint supports. 
Then, for each $k \in {\mathbb Z}$ and $i=1,2$,  
\begin{equation}
\int_X \vert \varphi_i \vert ^p \circ f^{-k} d\mu \leq Kt^{n} \int_X \vert \varphi_i \vert ^p \circ f^{-(k+n)} d\mu.\tag*{$(\bullet)$} 
\end{equation}
As $\varphi_1$ and $\varphi_2$ have disjoint supports, we have that, for  $m \in \mathbb Z$,  
\[\int_X \vert \varphi_1 + \varphi_2 \vert ^p \circ f^{-m} d\mu = \int_X \vert \varphi_1 \vert ^p \circ f^{-m} d\mu + \int_X \vert \varphi_2 \vert ^p \circ f^{-m} d\mu.\]
Now, by adding term by term in the inequalities $(\bullet)$, we obtain that, for each $k \in {\mathbb Z}$,
\[\int_X \vert \varphi_1 + \varphi_2 \vert ^p \circ f^{-k} d\mu 
\leq Kt^{n}  \int_X \vert \varphi_1 + \varphi_2 \vert ^p \circ f^{-(k+n)} d\mu.\]
Therefore, 
\[ \sup_{k \in {\mathbb Z}}{\left (\frac{\int_X \vert \varphi_1 + \varphi_2 \vert ^p \circ f^{-k} d\mu}{\int_X \vert \varphi_1 + \varphi_2 \vert ^p \circ f^{-(k+n)} d\mu}\right)} \leq Kt^n.\]
Hence, it follows that $\varphi_1 +  \varphi_2 \in \uc{(K,t)}$.
\end{proof}
\begin{prop} \label{prop:subW}
Let $(X,{\mathcal B},\mu, f)$ be a dissipative system of bounded distortion, generated by $W$. Let $H$ be the bounded distortion constant from Proposition~\ref{diststar} and $j \in \Z$.
\begin{enumerate}
    \item Let $\U (K,t) \in \{\uc{(K,t)}, \ud{(K,t)}\}$. If $\chi_{f^j(W)}$ is in $\U{(K,t)}$, then $\chi_{f^j(B)}$ is in $\U{(HK,t)}$, for all $B \subseteq W$ with $\mu(B) >0$.
    \item If $\chi_{f^j(W)}$ is in $ \ughm{(K,t)}$ for $j \ge 0$, then $\chi_{f^j(B)}$ is in $\ughm{(HK,t)}$ for all $B \subseteq W $with $\mu(B) >0$.
    \item If $\chi_{f^j(W)}$ is in $ \ughp{(K,t)}$ for $j <0$,  then $\chi_{f^j(B)}$ is in $\ughp{(HK,t)}$ for all $B \subseteq W $with $\mu(B) >0$.
\end{enumerate} 
\end{prop}
\begin{proof}
We do the proof for $\uc{(K,t)}$. Proofs for the rest are analogous. 

Assume the hypotheses, i.e., let $\chi_{f^j(W)} \in \uc{(K,t)}$
and $B \subseteq W$ with $\mu(B) >0$. By Proposition~\ref{diststar},  there exists a constant $0<H < + \infty$ such that, for every $s,l \in \mathbb Z$, 
\[ \dfrac{1}{H} \dfrac{\mu(f^{l+s}(W))}{\mu(f^s(W))} \leq \dfrac{\mu(f^{l+s} (B))}{\mu (f^s(B))} \leq H \dfrac{\mu(f^{l+s}(W))}{\mu(f^s(W))}.\]
Hence, in particular, for fixed $n \in \N$,
\[  \sup _{k  \in {\mathbb Z}} {\left( \frac{\mu(f^{k+j}(B))}{\mu(f^{k+j+n}(B))}\right)}  \leq H \sup _{k \in {\mathbb Z}} {\left(\frac{\mu(f^{k+j}(W))}{\mu(f^{k+j+n}(W))}\right)}.\] 
Then, 
\begin{eqnarray*}
\sup _{k \in {\mathbb Z}} {\left (\frac{\int_X \vert \chi_{f^j(B)} \vert ^p \circ f^{-k} d\mu}{\int_X \vert \chi_{f^j(B)} \vert ^p \circ f^{-(k+n)} d\mu}\right )} & = &  \sup _{k \in {\mathbb Z}} {\left( \frac{\mu(f^{k+j}(B))}{\mu(f^{k+j+n}(B))}\right)}  \\
& \leq & H  \sup _{k \in {\mathbb Z}} {\left(\frac{\mu(f^{k+j}(W))}{\mu(f^{k+j+n}(W))}\right)} \\
& = & H  \sup_{k \in \Z} \left(\frac{\int_X \vert \chi_{f^j(W)} \vert ^p \circ f^{-k} d\mu}{\int_X \vert \chi_{f^j(W)}\vert ^p \circ f^{-(k+n)} d\mu}\right) \\
& \leq & HKt^n,
\end{eqnarray*}
that is, $\chi_{f^j(B)}\in \uc{(HK,t)}$.

\end{proof}
The next proposition easily follows from the  well-known fact that the set of simple functions is dense in $L^p(X)$.
\begin{prop}\label{prop:simplefun}
Let $(X,{\mathcal B},\mu, f)$ be a dissipative system generated by $W$. Then, 
\begin{enumerate}
    \item $ \left \{ \sum_{i=0}^n a_i \chi_{B_i}: B_i \subseteq f^{j_{i}}(W),j_i \in \Z, \mu(B_i)>0, B_{i} \cap B_{i'} = \emptyset, i \neq i'  \right \}$ is dense in $L^p(X)$.
    \item $ \left \{ \sum_{i=0}^n a_i \chi_{B_i}: B_i \subseteq f^{j_{i}}(W),j_i <0, \mu(B_i)>0, B_{i} \cap B_{i'} = \emptyset, i \neq i'  \right \}$ is dense in $L_+$.
    \item $ \left \{ \sum_{i=0}^n a_i \chi_{B_i}: B_i \subseteq f^{j_{i}}(W),j_i  \ge 0, \mu(B_i)>0, B_{i} \cap B_{i'} = \emptyset, i \neq i'  \right \}$ is dense in $L_-$.
\end{enumerate}

\end{prop}
\begin{prop}\label{prop:DenseUs}
Let $(X,{\mathcal B},\mu, f)$ be a dissipative system of bounded distortion, generated by $W$. 
\begin{enumerate}
    \item Let $\U (K,t) \in \{\uc{(K,t)}, \ud{(K,t)}\}$. If $\chi_W \in \U(K,t)$, then $ \U(HK,t) = L^p(X)$.  
    \item If $\chi_W \in \ughm{(K,t)}$, then $\ughm{(HK,t)}=L_-$.
  \item If $\chi_{f^{-1}(W)} \in \ughp{(K,t)}$, then $\ughp{(HK,t)}=L_+$.
\end{enumerate}

\end{prop}
\begin{proof} 
(1) First, by Proposition \ref{prop:subW}, for every $B \subseteq W$ with $\mu(B) >0$,  $\chi_B$ is in $\U(HK, t)$. Then, by Proposition \ref{prop:compfj}, for every $i \in {\mathbb Z}$, for every $B \subseteq f^{i}(W)$ with $\mu(B) >0$,   $\chi_{B}$
is in $\U(HK,t)$.  Now, by Proposition \ref{prop:scalar} and Proposition \ref{finitesum}, 
\[ \left \{ \sum_{i=0}^n a_i \chi_{B_i}: B_i \subseteq f^{j_i}(W), j_i \in \Z, \mu(B_i)>0, B_{i} \cap B_{i'} = \emptyset, i \neq i'  \right \} \subseteq \U(HK,t).\]
Hence, by applying Proposition \ref{prop:simplefun}, and passing through limit, the conclusion follows.

The proofs of (2) and (3) are similar. We only show (2). First, by Proposition \ref{prop:subW}, for every $B \subseteq W$ with $\mu(B) >0$,  $\chi_B$ is in $\ughm(HK, t)$. Then, by Proposition \ref{prop:compfj}, for every $i \geq 0$, for every $B \subseteq f^{i}(W)$ with $\mu(B) >0$,   $\chi_{B}$
is in $\ughm(HK,t)$.  Now, by Proposition \ref{prop:scalar} and Proposition \ref{finitesum}, 
\[ \left \{ \sum_{i=0}^n a_i \chi_{B_i}: B_i \subseteq f^{j_i}(W), j_i \geq 0, \mu(B_i)>0, B_{i} \cap B_{i'} = \emptyset, i \neq i'  \right \} \subseteq \ughm(HK,t).\]
Hence, by applying Proposition \ref{prop:simplefun} and passing through limit, the conclusion follows. 
\end{proof}

{\em Proof of Theorem~\ref{thmSS}.}
\begin{proof} 
We first prove (1).
As Condition~$\hc{}$ is satisfied, by Proposition \ref{proprcW}, there exist $0<t <1$ and $K>0$ such that $\chi_{W} \in \uc{(K,t)}$. By Proposition \ref{prop:DenseUs}, $\uc{(HK,t)} = L^p(X)$, i.e, for all $\varphi \in L^p(X)$,
\begin{equation*}
 \sup _{k \in {\mathbb Z}} {\left (\frac{\int_X \vert \varphi \vert ^p \circ f^{-k} d\mu}{\int_X \vert \varphi \vert ^p \circ f^{-(k+n)} d\mu}\right )} \leq HKt^n .   
\end{equation*}
Plugging $k =-n$ in above, we have that, for all $n \in \N$,
\[ {\frac{\int_X \vert \varphi \vert ^p \circ f^{n} d\mu}{\int_X \vert \varphi \vert ^p d\mu}} \leq HK t^n. \] 
As
\[\int_X \vert \varphi \vert ^p \circ f^{n} d\mu= \Vert T_{f}^{n}(\varphi)\Vert_{p}^{p},\]
we have that
\[\frac{\Vert T_{f}^{n}(\varphi) \Vert_{p}^{p}}{\Vert \varphi \Vert_{p}^{p}}  \le HKt^n.\] Hence, 
\[\lim_{n \rightarrow \infty} \Vert T_{f}^{n} \Vert ^\frac{1}{n}\leq t < 1.\]
As the spectral radius of $T_{f} $ is $\lim_{n \rightarrow \infty} \Vert T_{f}^{n} \Vert ^\frac{1}{n}$, we have that $T_{f}$ is  a contraction.

Proof of (2) is analogous to the above case.

Finally, let us prove (3). 
Assume that $\gh{}$ holds. By Proposition \ref{proprcW}, there exist $0<t  <1$ and $K >0$ such that $\chi_{W} \in \ughm{(K,t)}$ and $\chi_{f^{-1}(W)} \in  \ughp{(K,t)}$. Then, by Proposition \ref{prop:DenseUs}, $\ughm{(HK,t)} = L_{-}$ and\\ $\ughp{(HK,t)}= L_{+}$. Hence, for all $\varphi \in L_{+}$,
\[\sup _{k \in -{{\mathbb N}}_{0}} \left ( \frac{\int_X  \vert \varphi \vert ^p \circ f^{-(k-n)}  d\mu}{\int_X  \vert \varphi \vert ^p \circ f^{-k}  d\mu} \right ) \leq HK t^n.\]
Plugging $k =0$ in above, we have that, for all $n \in \N$,
\[ {\frac{\int_X \vert \varphi \vert ^p \circ f^{n} d\mu}{\int_X \vert \varphi \vert ^p d\mu}} \leq HK t^n,\] 
implying that, for all $\varphi \in L_{+}$,
\[\frac{\Vert T_{f}^{n}(\varphi) \Vert_{p}^{p}}{\Vert \varphi \Vert_{p}^{p}}  \le HKt^n.\] 
Hence, the spectral radius of the restriction of $T_{f}$ to $L_{+}$, ${T_f}_{|_{L_+}}$, is less than or equal to $t <1$, 
therefore ${T_f}_{|_{L_+}}$ is a contraction. 

Similarly, as $\ughm{(HK,t)} = L_{-}$, we have that  for all $\varphi \in L_{-}$
\[ \inf_{k \in {\mathbb N}_{0} } {\left (\frac{\int_X  \vert \varphi \vert ^p \circ f^{-k}  d\mu}{\int_X  \vert \varphi \vert ^p \circ f^{-(k+n)}  d\mu}\right )} \geq HK \frac{1}{{t}^n}.\]  
Therefore, 
\[ \sup_{k \in {\mathbb N}_{0} } {\left (\frac{\int_X  \vert \varphi \vert ^p \circ f^{-(k+n)}  d\mu}{\int_X  \vert \varphi \vert ^p \circ f^{-k}  d\mu}\right )} \leq \frac{1}{HK} {t}^n.\]  
Plugging $k =0$ in above, we have that, for all $n \in \N$,
\[ \sup_{k \in {\mathbb N}_0 } {\left (\frac{\int_X  \vert \varphi \vert ^p \circ f^{-n}  d\mu}{\int_X  \vert \varphi \vert ^p  d\mu}\right )} \leq \frac{1}{HK} {t}^n.\] 
As
\[\int_X \vert \varphi \vert ^p \circ f^{-n} d\mu= \Vert {({T_{f}}^{-1})}^{n}(\varphi)\Vert_{p}^{p},\]
we have that, for all $\varphi \in L_{-}$,
\[\frac{\Vert {({T_{f}}^{-1})}^{n}(\varphi) \Vert_{p}^{p}}{\Vert \varphi \Vert_{p}^{p}}  \le HK{t}^n.\] 
Hence, the spectral radius of the restriction of ${T_{f}}^{-1}$ to $L_{-}$, ${T_f}^{-1}_{|_{L_-}}$, is less than or equal to $t <1$, therefore ${T_{f}}^{-1}_{|_{L_-}}$ is a contraction.

Thus, we have shown that $L^p(X) = L_+ \oplus L_-$, $T_f(L_+) \subseteq L_+$ and $T_f^{-1}(L_-) \subseteq L_-$, ${T_f}_{|_{L_+}}$ and ${T_f}_{|_{L_-}}^{-1}$ are contractions, i.e., $T_f$ is generalized hyperbolic.

\end{proof}

%@@@@@@@@@@@@@@@@@@@@
\subsection{Proof of Theorem~\ref{thmSN}}
%@@@@@@@@@@@@@@@@@@@

In general, a factor of a map with shadowing property does not have the shadowing property. A condition which guarantees this in the dynamics of compact metric spaces was given by Good and Meddaugh (\cite{GoodMeddaughIM2020}: Theorem 23). Below, we give a condition in the setting of linear dynamics which guarantees that factors of maps with the shadowing property have the shadowing property.  Using this result and the characterization of the shadowing property for weighted backward shifts, we arrive at the proof of the main theorem in this subsection. We begin by recalling the definition of a factor in linear dynamics.  

\begin{defn}\label{defSC}
Let $(X,S)$ and $(Y,T)$ be two linear dynamical systems. We say that {\em $T$ is a factor of $S$} if there exists a {\em factor map $\Pi$}, i.e., a linear, 
continuous, onto map $\Pi : X \rightarrow Y$ such that  $\Pi \circ S=T \circ \Pi$. Moreover, we say that {\em $\Pi$ admits a bounded selector} if the following condition holds. 
\[\exists L >0 \text{ s.t., } \forall y \in Y, \exists x \in {\Pi}^{-1}(y) \text{ with }\Vert x \Vert \leq L \Vert y \Vert. \] 
\end{defn}
\begin{lem} \label{shadowfactor}
Let $(X,S)$ and $(Y,T)$ be linear dynamical systems with a factor map $\Pi: X \rightarrow Y$ that admits a bounded selector. If $S$ has the shadowing property, then so does $T$. 
\end{lem}
\begin{proof} Let $L$ be a constant which witnesses that $\Pi$ has a bounded selector. 
We use the formulation of the shadowing property stated in Lemma \ref{LEM1}. Let $K$ be a constant associated with the fact that $S$ has the shadowing property. 
Let $\{y_n\}_{n \in \mathbb Z}$ be a bounded sequence in $Y$, with $\sup_{n \in {\mathbb Z}} \Vert y_{n} \Vert = M$. 
By hypothesis, we can take, for each $n \in {\mathbb Z}$, $x_{n} \in X$ such that $\Pi(x_{n})= y_{n}$, with $\Vert x_{n} \Vert \leq  L \Vert y_{n} \Vert$.
Then, $\{x_n\}_{n \in \mathbb Z}$ is a  sequence in $X$  with $\| x_n\| \le LM
$, $n \in \Z$.  As $S$ has the shadowing property, there is a sequence $\{s_n\}_{n \in \mathbb Z}$ in $X$ such that \[ \sup_{n \in \mathbb Z} \Vert s_n \Vert \leq K \sup_{n \in \mathbb Z} \Vert x_n \Vert \hspace{0.3 cm}\text{ and } \hspace{0.3 cm} s_{n+1}=S(s_n) + x_n, \text{ for all $n \in \mathbb Z$.}\] 
Setting $t_n = \Pi (s_n)$, we have that 
\[\sup_{n \in \mathbb Z} \Vert t_n \Vert
\leq  \Vert \Pi \Vert \sup_{n \in \mathbb Z} \Vert s_{n} \Vert  
\leq \Vert \Pi \Vert  K\sup_{n \in \mathbb Z} \Vert x_n \Vert 
\leq \Vert \Pi \Vert K  L \Vert y_{n} \Vert.\]
Moreover, applying $\Pi$ to the equation 
\[s_{n+1}=S(s_n) + x_n, \] 
we obtain that \[
t_{n+1}   =  T(t_{n}) + y_{n}.\]

Hence, we have proved that there exists a constant $C$, namely $C= \Vert \Pi \Vert  KL,$ such that, for every bounded sequence  $\{y_n\}_{n \in \mathbb Z}$ in $Y$, there is a sequence $\{t_n\}_{n \in \mathbb Z}$ in $Y$ such that 
\[ \sup_{n \in \mathbb Z} \Vert t_n \Vert \leq C \sup_{n \in \mathbb Z} \Vert y_n \Vert \hspace{0.3 cm}\text{ and } \hspace{0.3 cm} t_{n+1}=T(t_n) + y_n, \text{ for all $n \in \mathbb Z$},\] 
yielding that $T$ has the shadowing property. 
\end{proof}

\begin{lem} \label{factorBw} Let $(X,{\mathcal B},\mu, f)$ be a  dissipative system of bounded distortion generated by $W$. Consider the weighted backward shift $B_w$ on $\ell^p(\Z)$ with weights
\[w_{k} =  \left( \frac{\mu(f^{k-1}(W))}{\mu(f^{k}(W))}\right)^{\frac{1}{p}} .\] Then, $B_w$ is a factor of the map $T_f$ by a factor map $\Pi$ admitting a bounded selector. 
\end{lem}
\begin{proof} As $T_f$ is an invertible composition operator, we have that $f$ and $f^{-1}$ satisfy Condition~(\ref{condition}). Therefore, $0 < \inf _{n \in Z} |w_n|  \le \sup _{n \in Z} |w_n| < \infty $, implying that  $B_w$ is an invertible operator. 
We need to find a bounded linear surjective map $\Pi: L^p(X) \rightarrow \ell ^p({\mathbb Z})$ admitting a bounded selector
such that the diagram in Figure \ref{Figsemiconj2} commutes.

\begin{figure}[ht] 
\centering
\begin{tikzcd}
L^p(X)  \arrow[r, "T_f"] \arrow[d, "\Pi"]
& L^p(X)  \arrow[d, "\Pi"] \\ \ell^p({\mathbb Z}) \arrow[r,  "B_w"]
&\ell^p({\mathbb Z})
\end{tikzcd} 
\captionsetup{justification=centering,margin=2cm}
\caption{Factor map from $(L^p(X), T_f)$ to $(\ell^p({\mathbb Z}), B_w)$.}
\label{Figsemiconj2}
\end{figure} 

Let $\varphi \in L^p(X)$.  Define $\Pi(\varphi) = {\bf x} = \{x_{k}\}_{k \in {\mathbb Z}}$, where \[x_{k} = \dfrac{\mu(f^{k}(W))^{\frac{1}{p}}}{\mu(W)} \int_{W}  \varphi \circ f^{k} d \mu.\]
It is clear that $\Pi$ is linear. 
Now we show that $\Pi \circ T_{f} = B_{w} \circ \Pi$, that is, the diagram in Figure  \ref{Figsemiconj2} commutes. Indeed, for any $\varphi \in L^{p}(X)$, for any ${k \in \mathbb Z}$, 
\begin{align*}
{\left((\Pi \circ T_{f})(\varphi)\right)}_{k} &  ={\left(\Pi(\varphi \circ f)\right)}_{k} = \dfrac{\mu(f^{k}(W))^{\frac{1}{p}}}{\mu(W)} \int_{W} \left(\varphi \circ f \right) \circ  f^{k} d \mu\\
& =  \dfrac{\mu(f^{k}(W))^{\frac{1}{p}}}{\mu(W)} \int_{W} \varphi \circ f^{k+1} d \mu
\end{align*}
and, on the other hand, it is also the case that
\begin{align*}
{\left (B_{w} \circ \Pi\right)(\varphi)}_{k} & = w_{k+1} {\left( \Pi(\varphi)\right)}_{k+1} = \left( \frac{\mu(f^{k}(W))}{\mu(f^{k+1}(W))}\right)^{\frac{1}{p}}\frac{\mu(f^{k+1}(W))^{\frac{1}{p}}}{\mu(W)}  \int_{W} \varphi \circ f^{k+1} d \mu \\
& = \dfrac{\mu(f^{k}(W))^{\frac{1}{p}}}{\mu(W)} \int_{W} \varphi \circ f^{k+1} d \mu.
\end{align*}

We now show that $\Pi$ is a bounded operator with $\|\Pi\|_p \le H^{\frac{1}{p}} $, where $H$ is the bounded distortion constant in $(\Diamond \Diamond)$. In the proof, we will use the following version of Jensen Inequality: \[\left( \int _B g d\mu \right)^p \leq \mu(B)^{p-1} \int_B \vert g \vert ^p d\mu,\]
as well as the fact that if $\nu \ll \mu$, then 
\[\left \| \left. \frac{d \nu}{d \mu} \right |_{W} \right \|_{\infty} \le \sup_{ \substack{B \subseteq W \\ \mu (B) \neq 0} }\frac{ \nu(B)}{ \mu (B)}.\]
Let $\varphi \in L^{p}(X).$
 Then, 
\begin{align*}
{\Vert \Pi(\varphi) \Vert}_{p}^{p}  & =   {\Vert {\bf x} \Vert}_{p}^{p}  =  \sum_{k \in {\mathbb Z}} {\vert  x_{k}  \vert}^{p}   =  \sum_{k \in {\mathbb Z}} {\left(\frac{\mu(f^{k}(W))^{\frac{1}{p}}}{\mu(W)} \right)}^{p} {\left | \int_{W} \varphi \circ f^{k} d \mu  \right |}^{p} \\ 
&  \leq \sum_{k \in {\mathbb Z}} \frac{ \mu(f^{k}(W))}{\mu(W)^p} \mu(W)^{p-1} \int_W \vert \varphi \vert ^p \circ f^k  d\mu \\
& =  \sum_{k \in {\mathbb Z}} \frac{ \mu(f^{k}(W))}{\mu(W)}  \int_{f^k(W)} \vert \varphi \vert^p  d\mu f^{-k}\\
& = \sum_{k \in {\mathbb Z}} \frac{ \mu(f^{k}(W))}{\mu(W)}  \int_{f^k(W)} \vert \varphi \vert^p  \frac{d\mu f^{-k}}{d\mu} d\mu \\
& \leq \sum_{k \in {\mathbb Z}} \frac{ \mu(f^{k}(W))}{\mu(W)} \left. \left\|\frac{d\mu f^{-k}}{d\mu}\right |_{f^k(W)}\right  \|_{\infty} \int_{f^k(W)} \vert \varphi \vert^p d\mu \\
& \leq \sum_{k \in {\mathbb Z}} \frac{ \mu(f^{k}(W))}{\mu(W)}  \sup_{\substack{f^k(B), \\B\subseteq W}}\left(\frac{\mu( f^{-k}(f^k(B)))}{\mu(f^k(B))}\right) \int_{f^k(W)} \vert \varphi \vert^p d\mu \\
& = \sum_{k \in {\mathbb Z}} \frac{ \mu(f^{k}(W))}{\mu(W)}  \sup_{\substack{f^k(B), \\B\subseteq W}}\left( \frac{\mu(B)}{\mu(f^k(B))}\right) \int_{f^k(W)} \vert \varphi \vert^p d\mu \\
& \leq \sum_{k \in {\mathbb Z}} \frac{ \mu(f^{k}(W))}{\mu(W)} H \frac{\mu(W)}{\mu(f^k(W))} \int_{f^k(W)} \vert \varphi \vert^p d\mu\\
&= H \sum_{k \in {\mathbb Z}} \int_{f^k(W)} \vert \varphi \vert^p d\mu =H \Vert \varphi \Vert_p^p.
\end{align*} 
Hence, we have shown that 
\[{\Vert \Pi(\varphi) \Vert}_{p} \leq H^{\frac{1}{p}} {\Vert \varphi \Vert}_{p},\]
proving the continuity of $\Pi$.

We now prove that $\Pi$ admits a bounded selector with $L = 1$.  Let ${\bf x}= \{x_k\}_{k \in {\mathbb Z}} \in \ell ^p (\Z) $. We need to find $\varphi \in L^p(X)$ such that $\Pi(\varphi)={\bf x} $ with $\|\varphi\|_p \le \|{\bf x}\|_p$.
We let 
\[\varphi = \sum_{k \in {\mathbb Z}} \frac{x_k}{\mu(f^{k}(W))^{\frac{1}{p}}}  \chi_{f^{k}(W)}.\]
It is easy to verify that ${\|\varphi \|}_p = \|{\bf x} \|_p$.
Moreover,   
\begin{align*}
{\left(\Pi({\varphi})\right)}_{k} & =\frac{\mu(f^{k}(W))^{\frac{1}{p}}}{\mu(W)} \int_{W} \varphi \circ f^{k} d \mu \\
& =  \frac{\mu(f^{k}(W))^{\frac{1}{p}}}{\mu(W)}  \int_{W} \sum_{n \in {\mathbb Z}} \frac{x_n}{\mu(f^{n}(W))^{\frac{1}{p}}}  \chi_{f^{n}(W)} \circ f^{k} d \mu\\
& = \frac{\mu(f^{k}(W))^{\frac{1}{p}}}{\mu(W)} \sum_{n \in {\mathbb Z}} \int_{W}  \frac{x_n}{\mu(f^{n}(W))^{\frac{1}{p}}}  \chi_{f^{n}(W)} \circ f^{k} d \mu\\
& =  \frac{\mu(f^{k}(W))^{\frac{1}{p}}}{\mu(W)} \frac{x_{k}}{\mu(f^{k}(W))^{\frac{1}{p}}} \int_{W}  \chi_{f^{k}(W)} \circ f^{k} d \mu\\
& = \frac{x_{k}}{\mu(W)} \mu(W)= x_k.
\end{align*}
Hence, we have shown that $\Pi$ is a bounded, linear, surjective map admitting a bounded selector and such that $\Pi \circ T_f = B_w \circ \Pi$, completing the proof.
\end{proof}

{\em Proof of Theorem~\ref{thmSN}.}
Assume the hypotheses. By Lemma~\ref{factorBw}, we have that $B_w: \ell^p(Z) \rightarrow \ell^p(Z)$ is a factor of $T_f$ with 
\[w_{k} = \left( \frac{\mu(f^{k-1}(W))}{\mu(f^{k}(W))}\right)^{\frac{1}{p}}.\] Moreover, the factor map $\Pi$ which exhibits this admits a bounded selector. As $T_f$ has the shadowing property, by Lemma~\ref{shadowfactor}, we have that $B_w$ also has the shadowing property. By Theorem~\ref{theoSHADBW}, we have that Condition a), b) or c) of that theorem is satisfied. Now using the fact that 
\[w_k \ldots w_{k+n} = w_k \cdot \left( \frac{\mu(f^{k}(W))}{\mu(f^{k+1}(W)) }\right)^{\frac{1}{p}}\ldots \left( \frac{\mu(f^{k+n-1}(W))}{\mu(f^{k+n}(W))}\right)^{\frac{1}{p}} =w_k \cdot \left(\frac{\mu(f^{k}(W))}{\mu(f^{k+n}(W)) }\right)^{\frac{1}{p}},
\] and that
$ 0< \inf |w_k| \le  \sup |w_k| < \infty$,  
it is easy to check that Condition a) implies Condition~\ref{hc}, Condition b) implies Condition~\ref{hd} and Condition c) implies Condition~\ref{gh}. Indeed, for instance, if Condition a) holds, then

\[\lim_{n \rightarrow \infty}( \sup _{k \in {\mathbb Z}} \vert w_k \cdots w_{k+n}\vert ^{\frac{1}{n}} )<1,\] or, equivalently, 
\[\lim_{n \rightarrow \infty} \sup _{k \in {\mathbb Z}} \left \vert \left( \frac{\mu(f^{k}(W))}{\mu(f^{k+n}(W))}\right) ^{\frac{1}{p}}\right  \vert ^{\frac{1}{n}} =\lim_{n \rightarrow \infty} \sup _{k \in {\mathbb Z}} \left \vert w_k \cdot \left(\frac{\mu(f^{k}(W))}{\mu(f^{k+n}(W))}\right)^{\frac{1}{p}}\right  \vert ^{\frac{1}{n}}  < 1, \] 
implying that 
\[\lim_{n \rightarrow \infty} \sup _{k \in {\mathbb Z}} \left \vert \left( \frac{\mu(f^{k}(W))}{\mu(f^{k+n}(W))}\right) \right  \vert ^{\frac{1}{n}} < 1,\]
and yielding that Condition~\ref{hc} holds. 
Analogous arguments show the other two implications. 
%@@@@@@@@@@@@@@@@@@@@
\subsection{Proof of Theorem~\ref{thmRN}}
%@@@@@@@@@@@@@@@@@@@@

We will show in detail that Condition~\ref{hc} holds if and only if Condition~\ref{contractionRN} holds. Analogous arguments  will show that Condition~\ref{hd} and Condition~\ref{gh} hold if and only if Condition~\ref{dialationRN} and Condition~\ref{GHRN} hold, respectively. 

By the definition of Radon-Nikodym derivative, we have that, for all $i \in \Z$, 
\[ m_i \cdot \mu (W) \le \mu (f^i(W)) \le M_i \cdot \mu (W),\]
implying

\[  \frac{m_k}{M_{k+n}} \le \frac{\mu(f^{k}(W))}{\mu(f^{k+n}(W))} \le \frac{M_k}{m_{k+n}}.\]
Recall the hypothesis that $\frac{M_i}{m_i} < K$ for all $i \in {\mathbb Z}$. Hence, 
we have that $\frac{M_k}{m_{k+n}} \le K^2 \cdot \frac{m_k}{M_{k+n}}$, and putting it all together,
\[
 \frac{m_k}{M_{k+n}} \le \frac{\mu(f^{k}(W))}{\mu(f^{k+n}(W))} \le \frac{M_k}{m_{k+n}} \le K^2 \cdot \frac{m_k}{M_{k+n}}.
\] 
Now $\lim \limits_{n \rightarrow \infty} (K^2) ^{\frac{1}{n}} = 1$. This 
implies that 
\[\uplim_{n \rightarrow \infty} \sup _{k \in {\mathbb Z}}  \left (  \frac{m_k}{M_{k+n}} \right ) ^{\frac{1}{n}} = \uplim_{n \rightarrow \infty} \sup _{k \in {\mathbb Z}}  \left ( \frac{\mu(f^{k}(W))}{\mu(f^{k+n}(W))} \right ) ^{\frac{1}{n}} = \uplim_{n \rightarrow \infty} \sup _{k \in {\mathbb Z}}  \left (  \frac{M_k}{m_{k+n}} \right ) ^{\frac{1}{n}},
\]
completing the proof. \qedsymbol

\section{Open Questions}
We now make some final remarks and state some open questions. 

The following question addresses whether the condition of bounded distortion is necessary.
\begin{prob}
Suppose we have a dissipative system but we drop the hypothesis of bounded distortion. 
\begin{enumerate}
    \item Is it still the case that an operator is generalized hyperbolic if and only if it has the shadowing property?
    \item Does the characterization provided for the shadowing property still hold?
\end{enumerate}
\end{prob}
Next we inquire what happens in the direction orthogonal to dissipative systems.  
\begin{prob}
Suppose we work with purely conservative systems instead of dissipative systems.
\begin{enumerate}
    \item Is it still the case that an operator is generalized hyperbolic if and only if it has the shadowing property?
    \item Is there a natural characterization for the shadowing property?
    \item In particular, what happens if we consider linear operator induced by odometers as in \cite{BDDPOdometers}?
\end{enumerate}
\end{prob}
Finally, more generally, we have the following question.
\begin{prob}
In the arbitrary setting of linear dynamics, does shadowing imply generalized hyperbolicity?
\end{prob}
\bibliographystyle{siam}
\bibliography{biblio}

\Addresses

\end{document}